\documentclass[smallcondensed,nospthms]{svjour3}

\usepackage{mathptmx}
\usepackage{enumerate}
\usepackage{amssymb}
\usepackage{amsmath}
\usepackage{amsthm}

\def \IN{\mathbb N}

\def \IC{\mathbb C}
\def \ID{\mathbb D}

\numberwithin{equation}{section}

\newtheorem{theor}{Theorem}[section]
\newtheorem{cor}[theor]{Corollary}

\newtheorem{quest}[theor]{Question}

\theoremstyle{remark}
\newtheorem*{rem}{Remark}

\journalname{Potential Analysis}

\begin{document}

\title{On the harmonic measure and  capacity of rational lemniscates
\thanks{Research of second author supported by grants from NSERC and the Canada Research Chairs program.}}

\author{Stamatis Pouliasis \and Thomas Ransford}

\institute{Stamatis Pouliasis \at
D\'epartement de math\'ematiques et de statistique, Universit\'e Laval, Qu\'ebec (QC), Canada, G1V 0A6. \\
              \email{stamatispouliasis@gmail.com}            \\
             \emph{Present address:}   Faculty of Engineering and Natural Sciences, Sabanci University, Tuzla Istanbul 34956, Turkey
           \and
Thomas Ransford  \at
D\'epartement de math\'ematiques et de statistique, Universit\'e Laval, Qu\'ebec (QC), Canada, G1V 0A6. \\
              Tel.: +14186562131, ext 2738\\
              Fax: +14186562817\\
              \email{ransford@mat.ulaval.ca}
}

\date{}

\maketitle

\begin{abstract}
We study the lemniscates of rational maps. We prove a reflection principle for the harmonic measure of rational lemniscates and
we give estimates for their capacity and the capacity of their components. Also, we prove a version
of Schwarz's lemma for the capacity of the lemniscates of proper holomorphic functions.
\end{abstract}

\section{Introduction}

Let $\ID:=\{z\in\IC:|z|<1\}$ be the unit disc and let
$R$ be a rational function in the extended complex plane $\hat{\IC}$ with $R(\infty)=0$.
A set of the form
\[
\{z\in\hat{\IC}:|R(z)|=t\},
\qquad 0<t<\infty,
\]
is called a \textit{lemniscate} of $R$; we will also refer to sets of the form
\[
K_{t}:=\{z\in\hat{\IC}:|R(z)|\geq t\},
\qquad 0<t<\infty,
\]
as lemniscates of $R$. The properties of lemniscates of polynomials and rational functions have
been studied by many researchers. We mention here some recent results.

Anderson and Eiderman \cite{Anderson-Eiderman} proved that there exists an absolute constant $C>0$
such that, for the logarithmic derivative
\[
\frac{Q'_{n}(z)}{Q_{n}(z)}=\sum_{i=1}^{n}\frac{1}{z-z_{i}}
\]
of every polynomial $Q_{n}(z):=\prod_{i=1}^{n}(z-z_{i})$
of degree $n$, the inequality
\[
M\Big(\Big\{z\in\IC:\Big|\sum_{i=1}^{n}\frac{1}{z-z_{i}}\Big|>t\Big\}\Big)\leq\frac{C}{t}n\sqrt{\log n}
\]
holds, where $M$ denotes the 1-dimensional Hausdorff content.

Solynin and Williams \cite{Solynin-lemniscates} proved that, for each $n\geq1$, there exists a
constant $C(n)$ such that the inequality
\[
\frac{\lambda(\{z\in\hat{\IC}:|P(z)|\leq c\})}{\pi r^{2}(\{z\in\hat{\IC}:|P(z)|\leq c\})}\leq C(n)
\]
holds for every complex polynomial $P$ of degree $n$ and for every $c\in(0,+\infty)$, where $\lambda(E)$
is the area of $E$ and $r(E)$ is the inradius of $E$ (i.e. the supremum of the radii of open disks contained in $E$).

A map between two topological spaces $F:X\to Y$ is called \textit{proper} 
if the inverse image $F^{-1}(K)$ of every compact subset $K$ of $Y$ is a compact subset of $X$.
Dubinin \cite{Dubinin-lemniscate}, among other results, 
generalized a result of P\'olya for the area of a polynomial
lemniscate by proving the following inequality for a proper
holomorphic map $F$ from a domain $D$ onto a circular ring $\{z\in\IC:t_{1}<|z|<t_{2}\}$ ($0<t_{1}<t_{2}<+\infty$):
if $E$ is the union of all those connected components of $\hat{\IC}\setminus D$ whose boundaries contain points
corresponding, under the holomorphic function~$F$, to points on the circle $\{z\in\IC:|z|=t_{1}\}$ and $\infty\not\in E$,
then
\begin{equation}\label{Dubininlemniscateanisotita}
    \Big(\frac{t_{2}}{t_{1}}\Big)^{\frac{2}{n}}\leq\frac{\lambda(E\cup D)}{\lambda(E)}.
\end{equation}
Also, he proved that equality holds in (\ref{Dubininlemniscateanisotita}) if and only if $F(z)=c(z-a)^{n}$,
where $c$ and $a$ are arbitrary complex numbers.

Let $\Gamma$ be a $C^{\infty}$ Jordan curve in $\IC$ and let $G_{-}$ and $G_{+}$ denote
the bounded and unbounded component of $\hat{\IC}\setminus\Gamma$ respectively. From the
Riemann mapping theorem there exist conformal maps $\phi_{-}:\ID\mapsto G_{-}$ and $\phi_{+}:\hat{\IC}\setminus\overline{\ID}\mapsto G_{+}$
with $\phi_{+}(\infty)=\infty$ and $\phi'_{+}(\infty)>0$. It is well known that $\phi_{-}$, $\phi_{+}$
extend to $C^{\infty}$ diffeomorphisms on the closures of their respective domains. The map
$\phi_{+}^{-1}\circ\phi_{-}:\partial\ID\mapsto\partial\ID$ is called the \textit{fingerprint} of $\Gamma$.
Ebenfelt, Khavinson and Shapiro \cite{Ebenfelt-Khavinson-Shapiro}, among other results, proved that
the fingerprint of a polynomial lemniscate of degree $n$ is given
by the $n$-th root of a Blaschke product of degree $n$ and that conversely, any
smooth diffeomorphism induced by such a map is the fingerprint of a polynomial
lemniscate of the same degree. Younsi \cite{Younsi-fingerprints} generalized the above result to the
case of rational lemniscates.

For more results and applications of lemniscates we refer the reader to the books
\cite{Borwein-Erdelyi}, \cite{Rahman-Schmeisser} and \cite{Sheil-Small}.

The starting point of our work was a question posed by Younsi considering the capacity of the
components of the lemniscate of a good rational function. Following
\cite{Bourque-Younsi}, we will say that a rational function $R$ is $d$-\textit{good} ($d\in\IN$) if the degree of $R$ is $d$, if $R(\infty)=0$ and if the open set
$\Omega:=R^{-1}(\ID)$ is connected and bounded by $d$ disjoint analytic Jordan
curves $\gamma_{i}$, $i=1,\dots,d$.
Then $R$ has a simple pole $p_{i}$ on the bounded component of $\hat{\IC}\setminus\gamma_{i}$ for each $i$, and it can be written as
\[
R(z)=\sum_{i=1}^{d}\frac{a_{i}}{z-p_{i}},
\]
for some $a_{i}\in\IC\setminus\{0\}$, $i=1,\dots,d$.
Also, we will denote by $\zeta_{1}$,\dots,$\zeta_{d}$ the zeros of $R$ (repeated according to multiplicity).
We prove the following reflection principle for the harmonic measure of $\Omega$
and $\hat{\IC}\setminus \overline{\Omega}$: for every Borel set $E\subset\partial\Omega$,
\[
\sum_{i=1}^{d}\omega_{p_{i}}^{\hat{\IC}\setminus\overline{\Omega}}(E)=\sum_{j=1}^{d}\omega_{\zeta_{j}}^{\Omega}(E),
\]
where $\omega_{a}^{D}$ denotes the harmonic measure of an open set
$D\subset\hat{\IC}$ with respect to the point $a\in D$ (the above
equality is true for arbitrary rational functions, see Theorem
\ref{armonikometroritonsinartiseon}). Also, we show that the above
equality characterizes rational functions in the class of proper
holomorphic functions (see Theorem
\ref{xaraktirismosproperritonsinartiseonapoarmonikometro}). For the
logarithmic capacity of the component $K_{i}$ of
$K:=\hat{\IC}\setminus\Omega$ containing $p_{i}$, we give a new proof of the known result that
\[
{\rm{cap}}(K_{i})\geq|a_{i}|,\qquad i=1,\dots,d,
\]
and we show that there exists a constant $c$, depending just on the radius of injectivity of $R$
on $K_{i}$, such that
\[
{\rm{cap}}(K_{i})\leq c|a_{i}|.
\]
From \cite[Proposition 4.16, p. 114]{Tolsa-book} it follows that there exists an absolute
constant $C>0$ such that, for the lemniscate $K:=\{z\in\IC:|R(z)|\geq 1\}$ of every good rational function $R(z):=\sum_{i=1}^{d}(a_{i}/(z-p_{i}))$,
\begin{equation}\label{anisotitalemniscatecauchytransrormanalyticcapacity}
    \gamma(K)\leq C\sum_{i=1}^{d}|a_{i}|,
\end{equation}
where $\gamma$ denotes analytic capacity.
Younsi, motivated by considerations related to the semiadditivity property of analytic capacity, asked the following question:
\begin{quest}\label{malikquestion}
Given $d\geq 2$, does there exist a constant $C(d)>0$ with the following property: if $R(z):=\sum_{i=1}^{d}(a_{i}/(z-p_{i}))$
is a $d$-good rational function, then
\[
{\rm{cap}}(K_{i})\leq C(d)|a_{i}|,
\]
where $K_{i}$ is the component of the lemniscate
$K:=\{z\in\IC:|R(z)|\geq 1\}$ containing $p_{i}$?
\end{quest}
\noindent We answer negatively Younsi's question by giving examples of good rational functions of degree 3 such that the ratio ${\rm{cap}}(K_{i})/|a_{i}|$
can be arbitrarily large.
It is well known that, if $P(z):=\sum_{i=0}^{n}a_{i}z^{i}$ is a polynomial with $a_{n}\neq 0$ and $E$
is a compact subset of $\IC$,
then the logarithmic capacity of $P^{-1}(E)$ is given by
\[
{\rm{cap}}(P^{-1}(E))=\Big(\frac{{\rm{cap}}(E)}{|a_{n}|}\Big)^{\frac{1}{n}},
\]
see \cite[Theorem 5.2.5, p. 134]{Ransford}.
We give a lower estimate for the logarithmic capacity of the lemniscate $K$ of a good rational function $R$
taking into account the poles $\{p_{i}\}$ and the residues $\{a_{i}\}$ of $R$ by showing that
\[
{\rm{cap}}(K)\geq\Big[\prod_{\substack{i,j=1\\i\neq j}}^{d}|p_{i}-p_{j}|\prod_{i=1}^{d}|a_{i}|\Big]^{\frac{1}{d^{2}}},
\]
(see Theorems \ref{tiposgialogarithmikixoritikotita}, \ref{anofragmagialogarithmikixoritikotita} and
\ref{antiparadigmagiaxoritikotitasinistosas}).

Finally, if $f$ is a proper holomorphic function from a domain $\Omega$ to $\ID$ with \mbox{$f(\infty)=0$},
we prove a geometric version of Schwarz's lemma for the logarithmic capacity of the lemniscates
$K_{t}:=\hat{\IC}\setminus\{z\in\Omega:|f(z)|<t\}$
by showing that the function
\[
t\mapsto t^{\frac{1}{m(\infty)}}\cdot {\rm{cap}}(K_{t}),\qquad t\in(0,1),
\]
where $m(\infty)$ is the multiplicity of $f$ at $\infty$,
is non-decreasing and it is constant on a
neighborhood of $0$
(see Theorem \ref{Schwarzlemmaproperolomorfessinartiseis}).

\section{Notations and preliminaries}

\subsection{Good rational functions} Let $R(z):=\sum_{i=1}^{d}(a_{i}/(z-p_{i}))$ be a $d$-good rational function
and let $K_{i}$ be the component of the lemniscate $K:=\hat{\IC}\setminus\Omega$ containing $p_{i}$.
Then $R$ is injective on a neighborhood $V_{i}$ of $K_{i}$ and
we will denote by $R_{i}$ the restriction of $R$ on $V_{i}$, $i=1,\dots,d$.
We will denote by $D_{i}$ the interior of $K_{i}$.
Also, we let $Q_{i}:=1/{R_{i}}$ on $V_{i}$ and $P_{i}:=Q_{i}^{-1}$, and we note that
$P_{i}:\overline{\ID}\mapsto K_{i}$ is one to one and onto, $i=1,\dots,d$.

\subsection{Harmonic measure and logarithmic capacity} 
We will denote by $G_{D}(z,a)$, $z\in D$, and $\omega_{a}^{D}(E)$, $E\subset\partial D$,
the \textit{Green function} and the \textit{harmonic measure} of a
Greenian domain $D\subset\hat{\IC}$ with respect to the point $a\in D$. 
(Here \textit{Greenian} means simply that the domain possesses a Green function.
It is well known that a planar domain is Greenian if and only if 
its complement is of positive  logarithmic capacity.)
Also, we let $G_{D}(z,a):=0$ for $z\in\hat{\IC}\setminus D$ and
we note that, for domains $D$ that are regular for the Dirichlet
problem, $z\mapsto G_{D}(z,a)$ is a subharmonic function on
$\hat{\IC}\setminus\{a\}$.

The \textit{equilibrium energy} of a compact set $K\subset\IC$ is defined by
\[
I(K):=\inf_{\mu}\iint \log\frac{1}{|z-w|}d\mu(z)d\mu(w),
\]
where the infimum is taken over all Borel probability measures
$\mu$ supported on $K$. When $I(K)<+\infty$ the above infimum is attained by a unique probability measure
$\mu_{K}$ supported on $\partial K$, which is called the \textit{equilibrium measure} of $K$. The
\textit{logarithmic capacity} of $K$ is defined by
\[
{\rm{cap}}(K):=e^{-I(K)}.
\]
The logarithmic capacity is related to the Green function by the
following formula (\cite[Theorem 5.2.1 p. 132]{Ransford})
\begin{equation}\label{logarithmikixoritikotitakaisinartisigreen}
    {\rm{cap}}(K)=\exp\big(-\lim_{z\to\infty}\big(G_{\hat{\IC}\setminus
    K}(z,\infty)-\log|z|\big)\big).
\end{equation}
For more information about potential theory in the complex plane see e.g. \cite{Ransford}.

\subsection{Lindel\"of principle}
Let $f$ be a non-constant
holomorphic function on a Greenian domain $D$
such that $f(D)$ is Greenian.
The following inequality
is known as the \textit{Lindel\"of principle}:
\[
G_{f(D)}(w_{0},f(z))\geq\sum_{f(a)=w_{0}}m(a)G_{D}(a,z),
\]
where $z\in D$, $w_{0}\in f(D)$ and $m(a)$ is the
\textit{multiplicity} of the zero of $f(z)-f(a)$ at $a\in D$. It is
well known that, if $f$ is a proper holomorphic function from $D$ to
$f(D)$, then equality holds in the Lindel\"of principle (see e.g.
\cite{Heins Maurice}). For a characterization of the equality cases
in the Lindel\"of principle see \cite{Betsakos-Lindelof}.

\subsection{A majorization principle for harmonic measure under meromorphic functions}

We will use the following result of Dubinin for the behavior of
harmonic measure under certain meromorphic functions.

\begin{theor}[\protect{\cite[Theorem 2, p.753]{Dubinin-majorization}}]\label{DubininMajorization}
Let $D$ and $G$ be domains bounded by finitely many Jordan curves
and let $f$ be a meromorphic function on $D$ such that $f(\partial
D)\subset\IC\setminus G$. Suppose that the sets
$\gamma\subset\partial D$ and $f(\gamma)\subset\partial G$ consist
of finitely many open arcs, and positively oriented arcs from
$\gamma$ are mapped by $f$ to positively oriented arcs on $\partial
G$. If $w_{0}\in f(D)$, then
\begin{equation}\label{DubininInequality}
    \omega_{w_{0}}^{G}(f(\gamma))\leq\sum_{i=1}^{m}\omega_{z_{i}}^{D}(\gamma),
\end{equation}
where $z_{1},...,z_{m}$ are the zeros of the function $f-w_{0}$ if
$w_{0}\neq\infty$ and the zeros of $1/f$ if $w_{0}=\infty$ with
multiplicities taken into account. Equality in
(\ref{DubininInequality}) is attained if and only if $f$ is a proper
meromorphic function from $D$ to $G$ and the map $f:\gamma\mapsto
f(\gamma)$ is one to one.
\end{theor}

We note that inequality (\ref{DubininInequality}) and the equality
statement of Theorem \ref{DubininMajorization} remain true if we
replace $\gamma$ with an arbitrary Borel set $E\subset\gamma$.

\section{A reflection principle for harmonic measure}

In the following theorem we prove a reflection principle for the
harmonic measure of rational lemniscates, taking into account the
zeros and the poles of the rational function.

\begin{theor}\label{armonikometroritonsinartiseon}
Let $R$ be a rational function of degree $d$, let
$\zeta_{1},...,\zeta_{d}$ be the zeros and $p_{1},...,p_{d}$ be the
poles of $R$ and let $\Omega:=R^{-1}(\ID)$. Then
\begin{equation}\label{isotitaarmonikonmetron}
    \sum_{i=1}^{d}\omega_{p_{i}}^{\hat{\IC}\setminus\overline{\Omega}}(E)=\sum_{j=1}^{d}\omega_{\zeta_{j}}^{\Omega}(E),
\end{equation}
for every Borel set $E\subset\partial\Omega$.
\end{theor}

\begin{proof}
Let
$\Gamma:=\partial\Omega\setminus\{\zeta\in\partial\Omega:R'(\zeta)=0\}$.
Then there is a decomposition of $\Gamma$ by half-open arcs
$\Gamma_{1},\dots,\Gamma_{d}$ such that $R$ is injective on
$\Gamma_{i}$, $i=1,\dots,d$. Fix $i\in\{1,2,...,d\}$ and let $E$ be a
Borel subset of $\Gamma_{i}$. Let $A_{1}$ be the connected component
of $\Omega$ with $E\subset\partial A_{1}$ and let $B_{1}$ be the set
of zeros of $R$ on $A_{1}$. Also, let $A_{2}$ be the connected
component of $\hat{\IC}\setminus\overline{\Omega}$ with
$E\subset\partial A_{2}$ and let $B_{2}$ be the set of poles of $R$
on $A_{2}$. We note that $R$ is a proper meromorphic function from
$A_{1}$ to $\ID$ and from $A_{2}$ to
$\hat{\IC}\setminus\overline{\ID}$. From Theorem~\ref{DubininMajorization},
\begin{equation}\label{reflectionprincipleequality}
\begin{aligned}
\sum_{i=1}^{d}\omega_{p_{i}}^{\hat{\IC}\setminus\overline{\Omega}}(E)
&=\sum_{p\in B_{2}}m(p)\omega_{p}^{\hat{\IC}\setminus\overline{\Omega}}(E) 
=\omega_{\infty}^{\hat{\IC}\setminus\overline{\ID}}(R(E)) \\
&= \omega_{0}^{\ID}(R(E)) 
= \sum_{\zeta\in B_{1}}m(\zeta)\omega_{\zeta}^{\Omega}(E) 
= \sum_{j=1}^{d}\omega_{\zeta_{j}}^{\Omega}(E).
\end{aligned}
\end{equation}

For an arbitrary Borel set $E\subset\partial\Omega$, we may assume
that $E\subset\Gamma$, since harmonic measure does not change by
removing a finite number of points from $E$. Then, from the equality
\eqref{reflectionprincipleequality},
\[
\sum_{j=1}^{d}\omega_{\zeta_{j}}^{\Omega}(E) 
= \sum_{n=1}^{d}\sum_{j=1}^{d}\omega_{\zeta_{j}}^{\Omega}(E\cap \Gamma_{n}) 
= \sum_{n=1}^{d}\sum_{i=1}^{d}\omega_{p_{i}}^{\hat{\IC}\setminus\overline{\Omega}}(E\cap \Gamma_{n}) 
= \sum_{i=1}^{d}\omega_{p_{i}}^{\hat{\IC}\setminus\overline{\Omega}}(E).
\qedhere
\]
\end{proof}

\begin{rem}
Theorem~\ref{armonikometroritonsinartiseon} is a close relative of problems
on the proportionality of harmonic measures studied in \cite{AKS06} and \cite{So09}.
\end{rem}

In the next theorem we show that the reflection principle for
rational lemniscates proved above actually characterizes rational
functions among proper holomorphic functions.

\begin{theor}\label{xaraktirismosproperritonsinartiseonapoarmonikometro}
Let $\Omega$ be a finitely connected domain bounded by $d$ disjoint analytic Jordan curves $\gamma_{1},\dots,\gamma_{d}$,
with $\infty\in\Omega$.
Let $f$ be a proper holomorphic function of degree $d$ from $\Omega$ to $\ID$ and let $\zeta_{1},\dots,\zeta_{d}$ be its zeros.
Suppose further that, for every $i=1,\dots,d$ there exists $p_{i}$ in the bounded component of $\hat{\IC}\setminus\gamma_{i}$
such that
\begin{equation}\label{isotitaarmonikonmetrongiaxaraktirismo}
    \sum_{i=1}^{d}\omega_{p_{i}}^{\hat{\IC}\setminus\overline{\Omega}}(E)=\sum_{j=1}^{d}\omega_{\zeta_{j}}^{\Omega}(E)
\end{equation}
for every Borel set $E\subset\partial\Omega$. Then $f$ is a rational function.
\end{theor}

\begin{proof}
From the translation-invariance of harmonic measure we may assume
that $f(\infty)=0$. Let $D_{i}$ be the bounded component of
$\hat{\IC}\setminus\gamma_{i}$, $i=1,\dots,d$. We note that $f$ has
an analytic continuation on a neighborhood $V_{i}$ of $\partial
D_{i}$ and we may choose it such that the restriction $f_{i}$ of $f$
on $V_{i}$ is injective, $i=1,\dots,d$. Suppose that
(\ref{isotitaarmonikonmetrongiaxaraktirismo}) holds and let $m$ be
the normalized Lebesgue measure on the circle $\partial\ID$. From
Theorem~\ref{DubininMajorization},
\[
\sum_{j=1}^{d}\omega_{\zeta_{j}}^{\Omega}=\omega_{0}^{\ID}\circ
f_{i}=m\circ f_{i},\qquad \mbox{ on }V_{i},
\]
for every $i=1,\dots,d$. From
(\ref{isotitaarmonikonmetrongiaxaraktirismo}) we have
\[
\sum_{j=1}^{d}\omega_{\zeta_{j}}^{\Omega}=\omega_{p_{i}}^{D_{i}},\qquad \mbox{ on }V_{i}.
\]
Therefore $m\circ f_{i}=\omega_{p_{i}}^{D_{i}}$ on $V_{i}$,
$i=1,\dots,d$. Let $\phi_{i}$ be a conformal map of $D_{i}$ onto
$\hat{\IC}\setminus\overline{\ID}$ with $\phi_{i}(p_{i})=\infty$. We
may assume that $\phi_{i}$ has an analytic continuation on $V_{i}$.
From the conformal invariance of harmonic measure we obtain that
$\omega_{p_{i}}^{D_{i}}=\omega_{\infty}^{\hat{\IC}\setminus\overline{\ID}}\circ
\phi_{i}=m\circ \phi_{i}$. Therefore $m\circ f_{i}=m\circ \phi_{i}$
or $m\circ (f_{i}(\phi_{i}^{-1}(E)))=m(E)$, for every Borel set
$E\subset\partial\ID$. We obtain that $f_{i}\circ\phi_{i}^{-1}$ is a
diffeo\-morphism of $\partial\ID$ that preserves Lebesgue measure.
Therefore, $f_{i}\circ\phi_{i}^{-1}(\zeta)=\lambda\zeta$, or
$f_{i}(\zeta)=\lambda\phi_{i}(\zeta)$ for every
$\zeta\in\partial\ID$ and for some $\lambda\in\partial\ID$. From the
identity principle we have that
$f_{i}(\zeta)=\lambda\phi_{i}(\zeta)$ for every $\zeta\in V_{i}$.
Since this is true for every $i=1,\dots,d$, we obtain an extension
of $f$ as a meromorphic function on $\hat{\IC}$ with poles at
$p_{1},\dots,p_{d}$. Therefore, $f$ is a rational function.
\end{proof}

\section{Capacity of rational lemniscates}

In the following theorem we consider the logarithmic capacity of the lemniscate
$K$ of a $d$-good rational function $R$
and of its components $K_{i}$, $i=1,\dots,d$.

\begin{theor}\label{tiposgialogarithmikixoritikotita}
Let $R$ be a $d$-good rational function. Then
\begin{equation}\label{anisotitagialogarithmikixoritikotitaOmegaai}
   {\rm{cap}}(K_{i})\geq|a_{i}|,\qquad i=1,\dots,d,
\end{equation}
and
\begin{equation}\label{anisotitagialogarithmikixoritikotitaK}
   {\rm{cap}}(K)\geq\Big[\prod_{\substack{i,j=1\\i\neq j}}^{d}|p_{i}-p_{j}|\prod_{i=1}^{d}|a_{i}|\Big]^{\frac{1}{d^{2}}}.
\end{equation}
\end{theor}

\begin{rem}
Inequality \eqref{anisotitagialogarithmikixoritikotitaOmegaai} is a form of
Lavrent$'$ev's inequality on the product of conformal radii of two non-overlapping
simply connected domains (see e.g.\ \cite[p.223, Corollary~1]{Le75} or \cite{Du94}).
The proof given below is different.
We believe that inequality \eqref{anisotitagialogarithmikixoritikotitaK} is new,
though several similar inequalities were obtained in \cite[Chapter~3, \S6]{Le75}.
\end{rem}

\begin{proof} Let $D_{i}$ denote the interior of $K_{i}$, $i=1,\dots,d$.
Let $\mu:=\sum_{i=1}^{d}\omega_{p_{i}}^{D_{i}}$.
We have $I(\mu)=\sum_{i,j=1}^{d}I(\omega_{p_{i}}^{D_{i}},\omega_{p_{j}}^{D_{j}})$.
For $i\neq j$,
\begin{align*}
  I(\omega_{p_{i}}^{D_{i}},\omega_{p_{j}}^{D_{j}})
  &= \iint\log\frac{1}{|z-w|}d\omega_{p_{i}}^{D_{i}}(z)d\omega_{p_{j}}^{D_{j}}(w) \\
   &= \int\log\frac{1}{|p_{i}-w|}d\omega_{p_{j}}^{D_{j}}(w) \\
   &= \log\frac{1}{|p_{i}-p_{j}|}.
\end{align*}
For $i=j$, we note that
\[
H_{D_{i}}(z,w):=G_{D_{i}}(z,w)-\log\frac{1}{|z-w|}=-\int\log\frac{1}{|a-w|}d\omega_{z}^{D_{i}}(a)
\]
since $D_{i}$ is bounded (\cite[Theorem 4.4.7, p. 110]{Ransford}), and we obtain
\begin{align*}
I(\omega_{p_{i}}^{D_{i}})
&= \int\int\log\frac{1}{|z-w|}d\omega_{p_{i}}^{D_{i}}(z)d\omega_{p_{i}}^{D_{i}}(w) \\
   &= -\int H_{D_{i}}(p_{i},w)d\omega_{p_{i}}^{D_{i}}(w) \\
   &= -H_{D_{i}}(p_{i},p_{i}) \\
   &= -\lim_{z\to p_{i}}\Big[G_{D_{i}}(z,p_{i})-\log\frac{1}{|z-p_{i}|}\Big] \\
   &= -\lim_{z\to p_{i}}\Big[G_{\hat{\IC}\setminus\overline{\ID}}(R_{i}(z),\infty)-\log\frac{1}{|z-p_{i}|}\Big] \\
   &= -\lim_{z\to p_{i}}\log|(z-p_{i})R_{i}(z)| \\
   &= -\lim_{z\to p_{i}}\log\Big|\sum_{j=1}^{d}\frac{a_{j}(z-p_{i})}{z-p_{j}}\Big| \\
   &= \log\frac{1}{|a_{i}|}. \\
\end{align*}
Therefore,
\begin{equation}\label{isotitagiatinenergeiatonarmonikonmetronmetouspoloustisritissinartisis}
    I(\mu)=\sum_{\substack{i,j=1\\i\neq j}}^{d}\log\frac{1}{|p_{i}-p_{j}|}+\sum_{i=1}^{d}\log\frac{1}{|a_{i}|}.
\end{equation}
Inequality (\ref{anisotitagialogarithmikixoritikotitaOmegaai}) follows from
\[
I(K_{i})\leq I(\omega_{p_{i}}^{D_{i}})=\log\frac{1}{|a_{i}|}
\]
and inequality (\ref{anisotitagialogarithmikixoritikotitaK}) follows from
\[
I(K)\leq I(\frac{\mu}{d})=\frac{1}{d^{2}}I(\mu),
\]
since $\mu/d$ is a Borel probability measure on $K$.
\end{proof}

\begin{rem}
Let $R$ be a rational function of degree $d$
having simple poles $p_{i}$, $i=1,\dots,d$, and satisfying $R(\infty)=0$.
Let $a_{i}$ be the residue of $R$ at $p_{i}$, $i=1,\dots,d$.
Then, there exists $t_{0}\in(0,+\infty)$ such that, for every $t\in(t_{0},+\infty)$,
$z\mapsto R(z)/t$ is a $d$-good rational function.
Let $K_{i,t}$ be the component of the lemniscate $K_{t}:=\{z\in\hat{\IC}:|R(z)|\geq t\}$
containing $p_{i}$, $t\in(t_{0},+\infty)$.
Then,
from {\rm{Theorem \ref{tiposgialogarithmikixoritikotita}}},
we obtain that
\[
t\cdot {\rm{cap}}(K_{i,t})\geq|a_{i}|,\qquad i=1,\dots,d,
\]
and
\[
t^{\frac{1}{d}}\cdot {\rm{cap}}(K_{t})\geq\Big[\prod_{\substack{i,j=1\\i\neq j}}^{d}|p_{i}-p_{j}|\prod_{i=1}^{d}|a_{i}|\Big]^{\frac{1}{d^{2}}},
\]
for every $t\in(t_{0},+\infty)$.
\end{rem}

We will also examine the behavior of the logarithmic capacity of the
lemniscates $K_{t}$ for $t\in(0,t_{0})$. In fact, we will consider
proper holomorphic functions.
Burckel, Marshall, Minda, Poggi-Corradini and Ransford \cite{Burckel-Marshall-Minda-Poggi-Corradini-Ransford}
proved geometric versions of Schwarz's lemma for a holomorphic function $f$ on the unit disc $\ID$
by showing that the function
\[
r\mapsto\frac{T(f(r\ID))}{T(r\ID)},\qquad 0<r<1,
\]
is increasing, where $T(E)$ may be area, diameter or logarithmic
capacity of $E$. In the same article they asked about analogues of Schwarz's lemma for the dual situation of
holomorphic functions defined on a domain $\Omega$ onto the unit disc, where $\Omega$ satisfies some geometric restriction.
Dubinin's inequality (\ref{Dubininlemniscateanisotita}) is a result of this type considering the area of the
lemniscates of a proper holomorphic function from a domain in $\hat{\IC}$ to a circular ring. In the following theorem we prove a monotonicity
principle for the logarithmic capacity of the lemniscates of a proper holomorphic function from a finitely connected
domain to the unit disc.

\begin{theor}\label{Schwarzlemmaproperolomorfessinartiseis}
Let $f$ be a proper holomorphic function from a domain $\Omega\subset\hat{\IC}$ to $\ID$ such that
$\infty\in\Omega$ and $f(\infty)=0$. For every $t\in(0,1)$ we let
\[
\Omega_{t}:=\{z\in\Omega:|f(z)|<t\}
\]
and $K_{t}:=\hat{\IC}\setminus\Omega_{t}$. Then the function
\[
F(t):=t^{\frac{1}{m(\infty)}}\cdot {\rm{cap}}(K_{t}),\qquad t\in(0,1),
\]
is non-decreasing and there exists $t_{1}\in (0,1)$ such that
\begin{equation}\label{elpizonaeinaiiteleutaia}
    F(t)=|f^{(m(\infty))}(\infty)|^{\frac{1}{m(\infty)}},\qquad t\in(0,t_{1}).
\end{equation}
\end{theor}

\begin{proof} For $t\in(0,1)$, let $D_{t}$ denote the connected component of $\Omega_{t}$ that contains $\infty$,
and let $Z(t)$ be the set of zeros of $f$ in $D_{t}\setminus\{\infty\}$. Note that $f$ is a proper holomorphic
function from $D_{t}$ to $\ID$ and that
$Z(t_{1})\subset Z(t_{2})$ for $t_{1}\leq t_{2}$. Also, since the logarithmic capacity of a compact
set is equal to the logarithmic capacity of its outer boundary, we have that ${\rm{cap}}(K_{t})={\rm{cap}}(\hat{\IC}\setminus D_{t})$.
From the Lindel\"of principle we have that, for every $z\in D_{t}\setminus (Z(t)\cup\{\infty\})$,
\[
    \log\frac{t}{|z^{m(\infty)}f(z)|}=m(\infty)(G_{D_{t}}(z,\infty)-\log|z|)+\sum_{a\in Z(t)}m(a)G_{D_{t}}(z,a),
\]
and letting $z\to\infty$ we obtain
\[
{\rm{cap}}(K_{t})=\Big(\frac{|f^{(m(\infty))}(\infty)|}{t}\Big)^{\frac{1}{m(\infty)}}
    \exp\Big(\sum_{a\in Z(t)}\frac{m(a)}{m(\infty)}G_{D_{t}}(a,\infty)\Big).
\]
Then the monotonicity of the function $F$ follows from the positivity and the monotonicity
of the Green function and the monotonicity of the sets $Z(t)$. Also, equality
(\ref{elpizonaeinaiiteleutaia}) follows from the fact that there exists $t_{1}\in(0,1)$
such that $Z(t)=\varnothing$ for every $t\in(0,t_{1})$.
\end{proof}

We will make use of the fact that the restriction of a good rational function on one of the components
of its lemniscate is univalent. In this direction, the following well-known theorem about the growth
and the distortion of univalent functions in the unit disc will be useful.

\begin{theor}[\protect{\cite[Theorem 2.6, p. 33 and Corollary 7, p. 127]{Duren-univalentfunctions}}]\label{ektimisipilikoudiaforon}
If $f$ is holomorphic and univalent on $\ID$ such that
$f(0)=0$ and $f'(0)=1$, then
\[
\frac{r}{(1+r)^{2}}\leq|f(z)|\leq\frac{r}{(1-r)^{2}},
\]
for $|z|=r<1$ and
\[
\frac{1-r^{2}}{r^{2}}|f(z)f(w)|\leq\Big|\frac{f(z)-f(w)}{z-w}\Big|\leq\frac{|f(z)f(w)|}{r^{2}(1-r^{2})},
\]
for $|z|=|w|=r<1$ (for $z=w$ the difference quotient is to be interpreted as $f'(z)$).
\end{theor}

Using the above theorem, we obtain the following estimate for the logarithmic capacity of the
component $K_{i}$ of the lemniscate of a good rational function $R$ with respect
to the modulus of the corresponding residue $a_{i}$ of $R$ at $p_{i}$, under an injectivity assumption for $R$
on a neighborhood of $K_{i}$.

\begin{theor}\label{anofragmagialogarithmikixoritikotita}
Let $R$ be a $d$-good rational function and suppose that
\[
\{z\in\hat{\IC}:|z|\geq\frac{1}{r}\}\subset R(V_{i})
\]
for some $r>1$,
where $V_{i}$ is a neighborhood of $K_{i}$ and $R$ is injective on $V_{i}$, $i\in\{1,\dots,d\}$.
Then
\[
{\rm{cap}}(K_{i})\leq\frac{r^{6}}{(r^{2}-1)(r-1)^{4}}|a_{i}|.\]
\end{theor}

\begin{proof}
We note that $P_{i}=Q_{i}^{-1}$ is univalent on $\overline{D(0,r)}$, that $P_{i}(0)=p_{i}$ and that $P'_{i}(0)=a_{i}$.
Let
\[
M_{i}:=\sup_{z,w\in\partial K_{i}}\frac{|z-w|}{|Q_{i}(z)-Q_{i}(w)|}=\sup_{z,w\in\partial\ID}\frac{|P_{i}(z)-P_{i}(w)|}{|z-w|}.
\]
From the assumption for $r$ it follows that the function
\[
F_{i}(z):=\frac{P_{i}(rz)-p_{i}}{ra_{i}}
\]
is univalent in $\ID$ with $F_{i}(0)=0$ and $F'_{i}(0)=1$. From Theorem \ref{ektimisipilikoudiaforon} we have
\[
\Big|\frac{F_{i}(z)-F_{i}(w)}{z-w}\Big|\leq\frac{r^{4}|F_{i}(z)F_{i}(w)|}{r^{2}-1}\leq\frac{r^{6}}{(r^{2}-1)(r-1)^{4}},
\]
which implies that
\[
\frac{|P_{i}(rz)-P_{i}(rw)|}{|rz-rw|}
\leq\frac{r^{6}}{(r^{2}-1)(r-1)^{4}}|a_{i}|
\]
for $|z|=|w|=\frac{1}{r}<1$. Therefore,
\[
M_{i}\leq\frac{r^{6}}{(r^{2}-1)(r-1)^{4}}|a_{i}|.
\]

Let $\mu_{i}$ be the equilibrium measure of $K_{i}$ and consider the measure
\[
\nu_{i}(E):=\mu_{i}(Q_{i}^{-1}(E)),\qquad E\subset\partial\ID.
\]
Then,
\begin{align*}
   0=I(\overline{\ID}) &\leq I(\nu_{i}) \\
   &= \int\int\log\frac{1}{|z-w|}d\nu_{i}(z)d\nu_{i}(w) \\
   &= \int\int\log\frac{1}{|Q_{i}(z)-Q_{i}(w)|}d\mu_{i}(z)d\mu_{i}(w) \\
   &\leq \int\int\log\frac{M_{i}}{|z-w|}d\mu_{i}(z)d\mu_{i}(w) \\
   &=\log M_{i}+ I(K_{i}).
\end{align*}
Therefore,
\[
{\rm{cap}}(K_{i})\leq M_{i}\leq\frac{r^{6}}{(r^{2}-1)(r-1)^{4}}|a_{i}|.\qedhere
\]
\end{proof}

Let $R$ be a good rational function, let $p\in R^{-1}(\ID)$ and consider the rational
function
\[
R_{\epsilon}(z):=R(z)+\frac{\epsilon}{z-p},
\]
having residue $\epsilon>0$ at the
extra pole $p$. As a corollary of Theorem \ref{anofragmagialogarithmikixoritikotita}
we obtain an estimate for the rate of decrease of the logarithmic capacity of the component
of the lemniscate of $R_{\epsilon}$ that contains $p$, as $\epsilon\to 0$.

\begin{cor}\label{asymptotikisimperiforaxoritikotitassinistosas}
Let $R$ be a $d$-good rational function, let $p\in R^{-1}(\ID)$ and let
\[
R_{\epsilon}(z):=R(z)+\frac{\epsilon}{z-p},\qquad \epsilon>0,~ z\in\hat{\IC}.
\]
If $K_{\epsilon}$ is the component of the lemniscate $\{z\in\hat{\IC}:|R_{\epsilon}(z)|\geq 1\}$
of $R_{\epsilon}$ that contains $p$, then
\[
{\rm{cap}}(K_{\epsilon})=\mathcal{O}(\epsilon),\qquad\mbox{as }\, \epsilon\to 0.
\]
\end{cor}

\begin{proof}
We will denote by $K_{i}$, $i=1,\dots,d$, the components of the lemniscate
\[
K:=\{z\in\hat{\IC}:|R(z)|\geq 1\}
\]
of $R$.
Let $\bar{D}(\infty,r):=\{z\in\hat{\IC}:|z|\geq 1/r\}$, $r>0$. Since $R$ is a $d$-good rational function, there exist $r>1$ and a
neighborhood $V_{i}$ of $K_{i}$, $i=1,\dots,d$, such that
\[
\bar{D}(\infty,2r)\subset\bigcap_{i=1}^{d}R(V_{i}).
\]
Since
\[
|R_{\epsilon}(z)-R(z)|=\frac{\epsilon}{|z-p|},
\]
$R_{\epsilon}$ converges locally uniformly
to $R$ on $\IC\setminus\{p\}$ as $\epsilon\to 0$. Therefore, there exists $\epsilon_{0}>0$ such that,
for every $\epsilon<\epsilon_{0}$,
the rational function $R_{\epsilon}$ is  $(d+1)$-good  and
\[
\bar{D}(\infty,r)\subset\bigcap_{i=1}^{d}R_{\epsilon}(V_{i}).
\]
Since $R_{\epsilon}$ is a proper holomorphic function from $\hat{\IC}$ to $\hat{\IC}$ of degree $(d+1)$, for every $\epsilon<\epsilon_{0}$
there exists a neighborhood $V_{\epsilon}$ of $K_{\epsilon}$ such that
$\bar{D}(\infty,r)\subset R_{\epsilon}(V_{\epsilon})$ and $R_{\epsilon}$ is injective on
$R_{\epsilon}^{-1}(\bar{D}(\infty,r))\cap V_{\epsilon}$. From Theorem \ref{anofragmagialogarithmikixoritikotita}
we obtain that, for every $\epsilon<\epsilon_{0}$,
\[
{\rm{cap}}(K_{\epsilon})\leq\frac{r^{6}}{(r^{2}-1)(r-1)^{4}}\epsilon
\]
and the conclusion follows.
\end{proof}

Based on the previous results, one may ask if, given $d\geq2$, there exists a constant
\mbox{$C(d)>0$} such that ${\rm{cap}}(K_{i})\leq C(d)|a_{i}|,$ for every $d$-good rational function.
In the following theorem we show that the answer is no.

\begin{theor}\label{antiparadigmagiaxoritikotitasinistosas}
Let $a>0$ and $\eta\in(\frac{2}{3},1)$. For $p>1$ define
\[
R_{p}(z):=\frac{a}{z-p}+\frac{p-p^{\eta}}{z-ip}+\frac{p-p^{\eta}}{z+ip}.
\]
Then there exists $p_{0}:=p_{0}(a,\eta)$ such that, for all $p>p_{0}$,
\begin{enumerate}[\rm(i)]
  \item $R_{p}$ is a $3$-good rational function,
  \item the component of the lemniscate $\{z\in\hat{\IC}:|R_{p}(z)|\geq 1\}$ containing $p$ has logarithmic capacity
  at least $ap^{1-\eta}/8$.
\end{enumerate}
\end{theor}

\begin{proof}
(i) To show that $\{z\in\hat{\IC}:|R_{p}(z)|\geq 1\}$ has 3 components, it suffices to show that
each critical point $c$ of $R_{p}$ satisfies $|R_{p}(c)|<1$ (see \cite[Lemma 2.1]{Bourque-Younsi}).
We shall show that this is the case for all $p$ large enough.

The critical points of $R_{p}$ are the solutions $c$ of
\[
\frac{a}{(c-p)^{2}}+\frac{p-p^{\eta}}{(c-ip)^{2}}+\frac{p-p^{\eta}}{(c+ip)^{2}}=0.
\]
Simplifying, we obtain
\[
(c-p)^{3}(c+p)=-\frac{a}{2(p-p^{\eta})}(c^{2}+p^{2})^{2}.
\]
This has four roots (as expected), namely
\begin{center}
$c=-p+\mathcal{O}(1)$\,\, and\,\, $c=p-a^{1/3}\omega p^{2/3}+\mathcal{O}(p^{(\eta+1)/3})$\,\,\,\, ($p\to+\infty$),
\end{center}
where $\omega$ runs through the cube roots of unity. For the root near $-p$, we have
\[
R_{p}(c)=p^{\eta-1}-1+\mathcal{O}(p^{-1}),\qquad (p\to+\infty),
\]
and, for the roots near $p$, we have
\[
R_{p}(c)=1-p^{\eta-1}+\mathcal{O}(p^{-1/3}),\qquad (p\to+\infty).
\]
Since $\eta>2/3$, it follows that $|R_{p}(c)|<1$ for all $c$ and all sufficiently large $p$.

(ii) Consider points of the form $p+pt$, where $t\geq 0$. A simple calculation shows that $R_{p}(p+pt)\geq 1$
if and only if
\[
pt^{3}+(2p^{\eta}-a)t^{2}+2(p^{\eta}-a)t\leq 2a.
\]
For $0\leq t\leq a/(2p^{\eta})$, we have
\[
pt^{3}+(2p^{\eta}-a)t^{2}+2(p^{\eta}-a)t\leq a\Big(\frac{p^{1-3\eta}a^{2}}{8}\Big)+a\Big(\frac{a}{2p^{\eta}}\Big)+a,
\]
which is less than $2a$ if $p$ is sufficiently large. We conclude that, if $p$ is sufficiently large, then
$\{z\in\hat{\IC}:|R_{p}(z)|\geq 1\}$ contains the interval $[p,p+ap^{1-\eta}/2]$. Therefore, the logarithmic
capacity of the component of $\{z\in\hat{\IC}:|R_{p}(z)|\geq 1\}$ containing $p$ is at least as large as the capacity of $[p,p+ap^{1-\eta}/2])$, namely $ap^{1-\eta}/8$.
\end{proof}

\begin{acknowledgements}
The authors thank Malik Younsi for interesting discussions on the
subject, for posing Question \ref{malikquestion} and for pointing
out inequality
(\ref{anisotitalemniscatecauchytransrormanalyticcapacity}) and the
article \cite{Anderson-Eiderman}. Also, the authors thank Alexey
Lukashov for pointing out the article \cite{Dubinin-majorization},
and the  referee for drawing their attention to the references
\cite{AKS06,Du94,Le75,So09}.
\end{acknowledgements}

\end{document}